\documentclass[reqno]{amsart}
\usepackage{amsfonts, hyperref}
\usepackage{amssymb,enumerate}
\usepackage{amsthm}
\usepackage[all]{xy}

\theoremstyle{plain}
\newtheorem{thm}{Theorem}[section]
\newtheorem{cor}[thm]{Corollary}
\newtheorem{lem}[thm]{Lemma}
\newtheorem{prop}[thm]{Proposition}

\theoremstyle{definition}
\newtheorem{defn}[thm]{Definition}
\newtheorem{obs}[thm]{Remark}

\newtheorem{ex}[thm]{Example}

\theoremstyle{definition}

\DeclareMathOperator{\Td}{Rfd}
\DeclareMathOperator{\GF}{GF}
\DeclareMathOperator{\Gfd}{Gfd}

\DeclareMathOperator{\fd}{fd}
\DeclareMathOperator{\id}{id}
\DeclareMathOperator{\pd}{pd}
\DeclareMathOperator{\HH}{H}

\DeclareMathOperator{\Hom}{Hom}
\DeclareMathOperator{\Ext}{Ext}
\DeclareMathOperator{\M}{\mathcal{M}}
\DeclareMathOperator{\Tor}{Tor}
\DeclareMathOperator{\coker}{coker}

\DeclareMathOperator{\depth}{depth}

\DeclareMathOperator{\Ker}{Ker}
\DeclareMathOperator{\GI}{GI}
\DeclareMathOperator{\Gid}{Gid}
\DeclareMathOperator{\CF}{CF}

\newcommand{\X}{\mathcal{X}}

\newcommand{\op}{\mathrm{op}}

\newcommand{\Z}{\mathbb{Z}}
\newcommand{\ol}{\overline}
\newcommand{\xra}{\xrightarrow}
\newcommand{\vf}{\varphi}
\newcommand{\p}{\mathfrak{p}}
 
\renewcommand{\geq}{\geqslant}
\renewcommand{\leq}{\leqslant}

\numberwithin{equation}{thm}

\begin{document}

\bibliographystyle{amsplain}

\author{Sean Sather-Wagstaff}

\address{Mathematics Dept.,
NDSU Dept \# 2750,
PO Box 6050,
Fargo, ND 58108-6050
USA}

\email{sean.sather-wagstaff@ndsu.edu}

\urladdr{http://www.ndsu.edu/pubweb/\~{}ssatherw/}

\thanks{Sean Sather-Wagstaff was supported in part by a grant from the NSA.
Tirdad Sharif was supported by a grant from IPM (No. 83130311)}


\author{Tirdad Sharif}
\address{School of Mathematics
Institute for Studies in
Theoretical Physics and Mathematics
P. O. Box 19395-5746
Tehran, Iran}
\email{sharif@ipm.ir}

\author{Diana White}
\address{University of Colorado Denver,
Campus Box 170, 
PO Box 173364,
Denver, CO 80217-3364, USA}
\email{diana.white@ucdenver.edu}
\urladdr{http://math.ucdenver.edu/\~{}diwhite}

\title
{Relative Homology and Gorenstein Flat Dimension}


\begin{abstract}
We use the machinery of relative homological algebra to study modules of finite Gorenstein flat dimension. 
\end{abstract}

\maketitle

\section*{Introduction}

\noindent
\textbf{Convention.}
In this paper $R$ is a commutative noetherian ring with nonzero
identity, 
$\M(R)$ is the class of (unital) $R$-modules, and
$\X$ is a class of $R$-modules.

\

The point of this paper is to investigate the $R$-modules of finite Gorenstein flat dimension
through the methods of relative homological algebra. (See Section~\ref{sec01} for background information.)
The basic idea is not a new one: if an $R$-module $M$ admits a ``proper $\X$-resolution'' $X\to M$,
then $X$ is uniquely determined up to homotopy equivalence. This implies that the 
functors $\Tor^{\X}_i(M,-):=\HH_i(X\otimes_R-)$
and $\Ext_{\X}^i(M,-):=\HH^i(\Hom_R(X,-)$ are well-defined.
These functors provide non-trivial information about how $M$ interacts with the class $\X$.
For instance, vanishing of these functors can characterize when $M$ admits a bounded resolution
by modules in $X$. Motivating examples of this are in~\cite{AvM, EJ3, H2}; see also~\cite{sather:gcac, sather:crct}.

Each of the above cited  examples focuses primarily on deriving traditionally:
one uses classes of projective-like modules for Ext, and flat-like modules for Tor.
However, such restrictions are not required, as one sees in~\cite{aldrich:dfhrfc}
where $\Ext_{\mathcal F}^i(M,-)$ is studied for  the class $\mathcal F$ of flat $R$-modules. 

The current paper considers a similar construction $\Ext^i_{\GF}(M,-)$ where $\GF(R)$ is the class of Gorenstein flat modules.
Section~\ref{sec02} develops the foundations of $\Ext^i_{\GF}(M,-)$. 
For instance, we prove the following result showing that the vanishing of $\Ext^i_{\GF}(M,-)$ detects finiteness of $\Gfd_RM$.
Given the fact that we are deriving Hom in a non-traditional fashion, this result may be somewhat surprising.

\

\noindent
\textbf{Theorem~\ref{thm14}.}
\emph{For an $R$-module $M$ and an integer
$g\geq 0$, the following 
conditions are equivalent:
\begin{enumerate}[\rm(i)]
\item 
$\Gfd_R M\leqslant g$;
\item 
$\Ext_{\GF}^n (M,-) = 0$ for each $n>g$;
\item 
$\Ext_{\GF}^{g+1}(M,-) = 0$;
\item 
For each Gorenstein flat resolution
$G\to  M$ and each integer $n\geq g$ the cokernel
$\Omega^nG$ is Gorenstein flat;
\item 
M admits a proper Gorenstein flat resolution
$G\to  M$ such that the cokernel
$\Omega^gG$ is Gorenstein flat.
\end{enumerate}}

\

\noindent
We also study some aspects of $\Tor_i^{\GF}(M,-)$ in this section.

From~\cite{H}, we know that the class $\CF(R)$ of cotorsion $R$-modules of finite flat dimension
is useful for studying Gorenstein flat dimension. 
In Section~\ref{sec130711a}, we prove the following approximation result, \`a la~\cite{auslander:htmcma, H}.

\

\noindent
\textbf{Theorem~\ref{thm04}.}
\emph{If $M$ is an $R$-module
with $\Gfd_R M <\infty$, then $M$ has a monic CF-preenvelope   
$\varphi \colon M \to K$
such that
$\fd_R K  = \Gfd_R M$
and
$\coker \varphi$
is Gorenstein~flat.}

\

\noindent
We use this result in two distinct ways. 

First, in Section~\ref{sec130711a} we study some base-change behavior of Gorenstein flat dimension
and the associated relative (co)homology functors.
In particular, we provide a counterexample to (and corrected versions of) two results of~\cite{sahandi:dabf}.

Second, in Section~\ref{sec130705a} we study the homological dimension $\operatorname{CF-id}_RM$
and the relative derived functors $\Ext^i_{\CF}(-,M)$ defined in terms of
(proper) CF coresolutions of $M$.
For instance, we use these functors to characterize cotorsion properties of modules of finite Gorenstein flat dimension.

\

\noindent\textbf{Theorem~\ref{thm07}.}
\emph{Let $M$ be an R-module with
$\operatorname{Gfd}_RM<\infty$, and let $p\in\mathbb N$ be given. 
\begin{enumerate}[\rm(a)]
\item
The module $M$ is $p$-cotorsion with $\operatorname{fd}_RM<\infty$ if and only
if
$\operatorname{Ext}^n_{\CF}(G,M)=0=\operatorname{Ext}^n_R(G, M)$
for
$n>p$ and each $G\in \operatorname{GF}(R)$.
\item
Then $M$ is strongly $p$-cotorsion with $\operatorname{fd}_RM<\infty$ if
and only if
$\operatorname{Ext}^n_{\CF}(N,M)=0=\operatorname{Ext}^n_R(N,M)$ for
$n>p$ and for each R-module $N$ with
$\operatorname{Gfd}_RN<\infty$.
\end{enumerate}}

\

\noindent
The paper concludes with a brief discussion of the relation between 
the ``large restricted injective dimension'' of $M$ and $\Gid_RM$.

\section{Prerequisites} \label{sec01}

Here we outline foundational notions used throughout the paper.

\begin{defn}
In this paper, $R$-complexes (a.k.a., chain complexes of $R$-modules) are indexed homologically
\[
X=\cdots \to X_{\ell+1}
\xra{\partial^X_{\ell+1}} X_{\ell}
\xra{\partial^X_{\ell}} X_{\ell-1}\cdots
\]
except where we specify otherwise.
A morphism $X\to Y$ between $R$-complexes  is another word for ``chain map''.
\end{defn}

In this paper cotorsion modules play a major role.  

\begin{defn}
An $R$-module
$K$  is \emph{$n$-cotorsion} if
$\Ext_{R}^i(F,K)=0$ for each $i>n$ and each flat $R$-module $F$,
and $K$ is \emph{strongly $n$-cotorsion} if
$\Ext_{R}^i(X,K)=0$ for each $i>n$ and each $R$-module
$X$ of finite flat dimension.
The terms \emph{cotorsion} and \emph{strongly cotorsion}
are used for the case $n=0$.
We use the notation $\CF(R)$ for the class 
of $R$-modules consisting of the cotorsion modules of finite flat dimension.
\end{defn}

\begin{obs}\label{rmk130706a}
An $R$-module $M$ is in $\CF(R)$ if and only if there is an exact sequence
$$0\to F_n\to\cdots\to F_0\to M\to 0$$
such that each $F_i$ is flat and cotorsion, by~\cite[Proposition 4.6]{sather:abc}.
\end{obs}

The primary focus of this paper is relative homology with respect to
the class of Gorenstein flat modules, as
introduced by Enochs, Jenda, and Torrecillas~\cite{EJ1,enochs:gf}. 

\begin{defn}
An exact complex $F$ of flat
$R$-modules is a \emph{complete flat resolution} if
$J\otimes_{R}F$ is exact for each injective $R$-module $J$.  An
$R$-module is \emph{Gorenstein flat} if it is a cokernel of a
complete flat resolution.  The 
class 
of  Gorenstein flat $R$-modules is
denoted  $\GF(R)$.

An exact complex $E$ of injective $R$-modules is a
\emph{complete injective resolution} if $\Hom_{R}(J,E)$ is exact for each
injective $R$-module $J$. An $R$-module is \emph{Gorenstein
injective} if it is a cokernel of a complete injective resolution.
The class 
of $R$-Gorenstein injective
$R$-modules is denoted  $\GI(R)$.
\end{defn}

The notions of (pre)covers and (pre)envelopes
were defined by Enochs in~\cite{E1}.  For more information
see~\cite{Xu}.

\begin{defn}
An $R$-module homomorphism $\varphi\colon X\to  M$
with $X\in \X$
is  an
\emph{$\X$-precover} if, for each $Y\in \X$, the following map is surjective:
$$
\Hom_R(Y,\varphi)\colon \Hom_R(Y,X)\to   \Hom_R(Y,M).
$$
An $\X$-precover $\varphi\colon X\to M$
is an \emph{$\X$-cover} if every endomorphism $f\colon X\to   X$
such that
$\varphi f = \varphi$
is an automorphism.

A homomorphism
$\varphi\colon M \to   X$
with $X\in\X$ is  an \emph{$\X$-preenvelope} of $M$ if,
for each $Y\in \X$,
the following map is
surjective: 
$$
\Hom_{R}(\varphi,Y)\colon\Hom_{R}(X,Y)\to  \Hom_{R}(M,Y).
$$
An
$\X$-preenvelope
$\varphi\colon M\to   X$
is an \emph{$\X$-envelope} if each endomorphism
$f\colon X\to   X$
such that
$f\varphi =\varphi$
is an automorphism.
\end{defn}

\begin{obs}
If $\mathcal X$ contains all projective $R$-modules, then every $\X$-precover is surjective.
On the other hand, if $\mathcal Y$ is a class of $R$-modules containing all injective $R$-modules, then every
$\mathcal Y$-preenvelope is injective.
For instance, this applies when $\X=\operatorname{GF}(R)$ and $\mathcal Y=\operatorname{GI}(R)$.
\end{obs}

\begin{defn}
We write pd for projective dimension, fd for flat
dimension, and id for injective dimension.
A \emph{Gorenstein flat resolution} of an $R$-module $M$ is a complex of Gorenstein flat 
modules
$$ 
G : \cdots\to
G_g\to \cdots G_1\to G_0\to 0
$$
with an augmentation map $G_0\to M$ such that the
augmented complex
$$ G^{+}: \cdots\to
G_g\to \cdots G_1\to G_0\to
M\to   0
$$
is exact. 
A projective or flat resolution of $M$ is a Gorenstein flat resolution.
We  write $G\to M$ to denote a
Gorenstein flat resolution of $M$, and we set
$$\operatorname{Gfd}_R M =\inf\{\sup\{i\geq 0\mid G_i\neq 0\}\mid\text{$G\to M$ is a Gorenstein flat resolution}\}.
$$
The notions of 
\emph{Gorenstein injective resolution} and \emph{CF coresolution} are  dual, 
with
\begin{gather*}
\operatorname{Gid}_R M =\inf\{\sup\{i\geq 0\mid G_{-i}\neq 0\}\mid\text{$M\to G$ is a Gorenstein injective resolution}\}\\
\operatorname{CF-id}_R M =\inf\{\sup\{i\geq 0\mid G_{-i}\neq 0\}\mid\text{$M\to G$ is a CF coresolution}\}.
\end{gather*}
\end{defn}

\begin{defn}
A complex $X$ of $R$-modules
is  \emph{GF-exact} if $\Hom_R(G,X)$ is
homologically trivial for each $G\in \GF(R)$.  Since 
$R$ is in $\GF(R)$, each GF-exact complex is
exact. A \emph{proper Gorenstein flat resolution} of an
$R$-module $M$ is a Gorenstein flat resolution $G$ of $M$ such that
the augmented complex $G^+$ is GF-exact.  
The notions of  \emph{CF-coexact} and
\emph{proper CF coresolution} are defined dually.
\end{defn}

\begin{obs} \label{rmk01}
From~\cite[3.3]{enochs:egfc} we know that every $R$-module has a GF-cover.
Hence, every $R$-module has a proper Gorenstein flat resolution.
From~\cite[Theorem 3]{bican:amhfc} we know that every $R$-module $M$ has a  flat cover $F\xra\phi M$,
and Wakamatsu's Lemma~\cite[7.2.3]{EJ3} implies that $\Ker(\phi)$ is cotorsion.
In particular, if $M$ is cotorsion, then $F\in \operatorname{CF}(R)$.
\end{obs}

The next  results indicate the power of proper 
resolutions.  The proofs are nearly identical to those 
of~\cite[(4.3) and (4.5)]{AvM}, so we omit them here;
see also~\cite[(2.2)]{H2}. 

\begin{lem} \label{lem01} 
Let $M,M'$ be $R$-modules and assume that $M'$ admits a 
proper Gorenstein flat resolution $\gamma'\colon G' \to M'$.  
Fix a Gorenstein flat resolution 
$\gamma\colon G \to M$ and projective resolutions
$\pi\colon P\to M$ and $\pi'\colon P'\to M'$. 
\begin{enumerate}[\rm(a)]
\item \label{item26}
There exist morphisms of complexes $\phi\colon P\to G$ and
$\phi'\colon P'\to G'$ such that $\pi=\gamma\phi$ and 
$\pi'=\gamma'\phi'$.  The morphisms $\phi$ and $\phi'$ are unique up 
to homotopy.
\item \label{item27}
For every homomorphism
$\mu\colon M\to M'$ there is a morphism of complexes
$\widetilde{\mu}\colon G\to G'$, unique up to homotopy, making the
right-hand square of the following diagram commute.
$$\xymatrix{
P \ar[r]^{\phi}\ar[d]_{\ol{\mu}} & G \ar[r]^{\gamma} \ar[d]_{\widetilde{\mu}} & M \ar[d]_{\mu} \\
P' \ar[r]^{\phi'} & G' \ar[r]^{\gamma'} & M'.}$$
For each such $\widetilde{\mu}$ there exists a morphism $\ol{\mu}\colon 
P\to P'$, unique up to homotopy, making the left-hand square of 
the diagram commute up to homotopy.
If $\mu=id_M$ and $G$ is also proper, then $\widetilde{\mu}$ and
$\overline{\mu}$ are homotopy
equivalences.
\end{enumerate}
\end{lem}

\begin{lem} \label{lem03}
Let $0\to M\xra{\mu} M'\xra{\mu'} M''\to 0$ be a GF-exact 
sequence of $R$-modules.  
Fix proper Gorenstein flat resolutions 
$\gamma\colon G \to M$ and $\gamma''\colon G'' \to M''$,
and fix projective resolutions
$\pi\colon P\to M$ and $\pi'\colon P''\to M''$. 
There exists a commutative diagram of morphisms
$$\xymatrix{
0 \ar[r] & P\ar[d]_{\phi} \ar[r]^{\ol{\mu}} & P' \ar[d]_{\phi'}\ar[r]^{\ol{\mu}'} 
& P'' \ar[r]\ar[d]_{\phi''} & 0 \\
0 \ar[r] & G\ar[d]_{\gamma} \ar[r]^{\widetilde{\mu}} & G' \ar[r]^{\widetilde{\mu}'} \ar[d]_{\gamma'}
& G''  \ar[r]\ar[d]_{\gamma''} & 0 \\
0 \ar[r] & M \ar[r]^{\mu} & M' \ar[r]^{\mu'} 
& M'' \ar[r] & 0}$$ 
where the top and middle rows are degreewise split exact, $\gamma'$ is 
a proper Gorenstein flat resolution, $\pi'=\gamma'\phi'$ is a projective 
resolution, and $\pi=\gamma\phi$ and 
$\pi''=\gamma''\phi''$.  
\end{lem}

The following tools of Christensen, Foxby and Frankild~\cite{Chr,F3}
are useful for tracking Gorenstein homological dimensions.

\begin{defn}
Let $M$ be an $R$-module.
The \emph{large restricted flat-dimension}
and \emph{large restricted injective-dimension} of $M$ are
\begin{align*}
\Td_R M &= \sup\{i\mid \text{$\Tor_i ^R(L,M) \neq 0$ for some $R$-module $L$
with $\fd_R L<\infty$} \}\\
\operatorname{Rid}_RM&=\operatorname{sup}\{i\mid\text{$\operatorname{Ext}^i_R(T,M)\neq 
0$ for some $R$-module $X$ with $\mathrm{pd}_RX<\infty$}\}.
\end{align*}
\end{defn}

\section{Relative Homological Algebra with Gorenstein flat Modules} \label{sec02}

In this section we investigate relative homological algebra with
respect to the class of Gorenstein flat modules in the sense of
Eilenberg and Moore~\cite{EM} and Enochs and Jenda~\cite{EJ3}.

\begin{defn}\label{defn130702a}
Let $M,M',N,N'$ be $R$-modules such that $M$ and  $M'$ admit
proper Gorenstein flat resolutions
$G\to   M$ and $G'\to   M'$. 
For each
$n\in\Z$
and every $R$-module $N$, we consider the $n$-th \emph{relative cohomology modules} 
\begin{align*}
 \Ext_{\GF}^{n}(M,N) &= \HH^n(\Hom_R(G,N)) =\HH_{-n}(\Hom_R(G,N)) \\
 \Tor_{n}^{\GF}(M,N) &= \HH_n(G\otimes_{R}N). 
\end{align*}
For each $R$-linear map 
$\mu\colon M\to M'$, let $\widetilde\mu\colon G\to G'$ be a morphism as 
in Lemma~\ref{lem01}.  Applying 
the functors $\Hom_R(-, N)$ and
$-\otimes_R N$ to $\widetilde\mu$ and taking homology yield
homomorphisms
\begin{align*}
\Ext_{\GF}^n(\mu,N)&\colon \Ext_{\GF}^n(M',N) \to \Ext_{\GF}^{n} (M,N) \\
\Tor_n^{\GF}(\mu,N)&\colon \Tor_n^{\GF}(M,N)\to \Tor_n^{\GF}(M',N).
\end{align*}
Furthermore, an $R$-linear map $\tau\colon N\to N'$ gives rise to more 
homomorphisms
\begin{align*}
\Ext_{\GF}^n(M,\tau)&\colon \Ext_{\GF}^n(M,N) \to \Ext_{\GF}^{n} (M,N') \\
\Tor_n^{\GF}(M,\tau)&\colon \Tor_n^{\GF}(M,N)\to \Tor_n^{\GF}(M,N').
\end{align*}
\end{defn}

\begin{obs}\label{rmk130702a}
Continue with the notation of Definition~\ref{defn130702a}.
By Lemma~\ref{lem01} the modules $\Ext_{\GF}^{n}(M,N)$ and
$\Tor_{n}^{\GF}(M,N)$
are independent of the choice of
resolution and  that the homomorphisms $\Ext_{\GF}^{n}(\mu,N)$ and
$\Tor_{n}^{\GF}(\mu,N)$ are independent 
of the resolutions and lifts chosen.
It is straightforward to show that the maps $\Ext_{\GF}^{n}(M,\tau)$ and
$\Tor_{n}^{\GF}(M,\tau)$
are 
well-defined as well.
Furthermore, it is not difficult to check that these constructions are functorial and 
compatible, in the sense
that the following diagrams commute
$$\xymatrix@C=15mm{\Ext_{\GF}^{n} (M',N) \ar[d]_{\Ext_{\GF}^n(M',\tau)}\ar[r]^{\Ext_{\GF}^n(\mu,N)} & \Ext_{\GF}^{n} (M,N) \ar[d]^{\Ext_{\GF}^n(M,\tau)}\\
\Ext_{\GF}^{n} (M',N') \ar[r]^{\Ext_{\GF}^n(\mu,N')} & \Ext_{\GF}^{n} (M,N') 
}$$
$$\xymatrix@C=15mm{
\Tor_n^{\GF}(M,N)\ar[d]_{\Tor_n^{\GF}(M,\tau)} \ar[r]^{\Tor_n^{\GF}(\mu,N)} & \Tor_n^{\GF}(M',N)\ar[d]^{\Tor_n^{\GF}(M',\tau)} \\
\Tor_n^{\GF}(M,N') \ar[r]^{\Tor_n^{\GF}(\mu,N')} & \Tor_n^{\GF}(M',N'). 
}$$
\end{obs}

\begin{prop} \label{thm12}
Let $n\geq 0$. 
\begin{enumerate}[\rm(a)]
\item \label{item16}
The constructions described above give 
well-defined additive functors
\begin{align*}
\Ext^n_{\GF}(-,-)&\colon \M(R)^{\op}\times\M(R)\to\M(R) \\
\Tor_n^{\GF}(-,-)&\colon \M(R)\times\M(R)\to\M(R). 
\end{align*}
\item \label{item17}
For all $R$-modules $M$ and $N$, there are natural isomorphisms
$\Ext^0_{\GF}(M,N)\cong\Hom_R(M,N)$ and $\Tor_0^{\GF}(M,N)\cong M\otimes_R N$.
\end{enumerate}
\end{prop}

\begin{proof}
Part~\eqref{item16} follows from Remark~\ref{rmk130702a}, and part~\eqref{item17} is a consequence of the left and right 
exactness of Hom and tensor product, respectively.
\end{proof}

Our next results describe exactness properties for these functors as in~\cite{AvM}.

\begin{prop} \label{thm01}
Given an $R$-module $M$ and a GF-exact sequence of modules
$0\to N_2\xra{\alpha_2}   N_1\xra{\alpha_1} N_0\to 0$ 
there is a long exact sequence 
$$
\cdots
\Ext_{\GF}^{n}(M,N_2)
\to
\Ext_{\GF}^{n}(M,N_1)
\to
\Ext_{\GF}^{n}(M,N_0)
\xra{\vartheta^{n}_{MN}}
\Ext_{\GF}^{n+1}(M,N_2)
\cdots
$$
where the unmarked maps are $\Ext_{\GF}^{n}(M,\alpha_2)$ and 
$\Ext_{\GF}^{n}(M,\alpha_1)$, respectively.
This sequence is natural in both arguments, in the following sense.
\begin{enumerate}[\rm(a)]
\item \label{item18}
Given an $R$-module homomorphism 
$\mu\colon M\to M'$, the next diagram commutes
$$\xymatrix@C=4.5mm{\cdots  \Ext_{\GF}^{n}(M,N_2)\ar[d] \ar[r] & \Ext_{\GF}^{n}(M,N_1)\ar[d] 
\ar[r] & \Ext_{\GF}^{n}(M,N_0) \ar[d]\ar[r] & 
\Ext_{\GF}^{n+1}(M,N_2)\ar[d]  \cdots \\
\cdots  \Ext_{\GF}^{n}(M',N_2)  \ar[r] & \Ext_{\GF}^{n}(M',N_1) 
\ar[r] & \Ext_{\GF}^{n}(M',N_0) \ar[r] & 
\Ext_{\GF}^{n+1}(M',N_2)  \cdots}$$
where each vertical map is the appropriate 
$\Ext^i_{\GF}(\mu,N_j)$.
\item \label{item19}
Given a commutative diagram with both rows GF-exact
$$\xymatrix{
0 \ar[r] & N_2 \ar[r]\ar[d]_{\beta_2} & N_1 \ar[r]\ar[d]_{\beta_1} & N_0 \ar[r]\ar[d]_{\beta_0} & 0 \\
0 \ar[r] & N_2' \ar[r] & N_1' \ar[r] & N_0' \ar[r] & 0}$$
the following diagram 
commutes
$$\xymatrix@C=6mm{
\cdots  \Ext_{\GF}^{n}(M,N_2) \ar[r] \ar[d]& \Ext_{\GF}^{n}(M,N_1) \ar[r]\ar[d] & \Ext_{\GF}^{n}(M,N_0) \ar[r]\ar[d] &
\Ext_{\GF}^{n+1}(M,N_2) \cdots \\
\cdots  \Ext_{\GF}^{n}(M,N_2') \ar[r] & \Ext_{\GF}^{n}(M,N_1') \ar[r] & \Ext_{\GF}^{n}(M,N_0') \ar[r]&
\Ext_{\GF}^{n+1}(M,N_2')  \cdots}$$
where each vertical map is the appropriate 
$\Ext^i_{\GF}(M,\beta_j)$.
\end{enumerate}
\end{prop}

\begin{proof}
Let $G\to M$ be a proper Gorenstein flat resolution.  
Since the given sequence is GF-exact the following sequence of 
morphisms is exact.
$$
0\to   \Hom(G,N_2)
 \xra{\Hom_R(G,\alpha_2)} \Hom_R(G,N_1)
 \xra{\Hom_R(G,\alpha_1)} \Hom_R(G,N_0)
 \to 0
$$
The associated long exact sequence is the desired one.

\eqref{item18}
Let $G'\to M'$ be a proper Gorenstein flat resolution and $\widetilde{\mu}\colon 
G\to G'$ a morphism as in Lemma~\ref{lem01}.  The desired diagram 
comes from taking the long exact sequences of the rows of the 
following commutative diagram 
$$\xymatrix@C=12mm{
0 \ar[r] & \Hom(G,N_2) \ar[r]^{\Hom(G,\alpha_2)} \ar[d]& \Hom(G,N_1)
\ar[d]\ar[r]^{\Hom(G,\alpha_1)} & \Hom(G,N_0) \ar[r] & 0 \\
0 \ar[r] & \Hom(G',N_2) \ar[r]^{\Hom(G',\alpha_2)} & \Hom(G',N_1)
\ar[r]^{\Hom(G',\alpha_1)} & \Hom(G',N_0) \ar[r] & 0}$$ 
where each vertical map is  the appropriate $\Hom(\widetilde{\mu},N_j)$.

\eqref{item19}
The desired diagram comes from the following one
$$\xymatrix@C=13mm{
0 \ar[r] & \Hom(G,N_2) \ar[d]\ar[r]^{\Hom(G,\alpha_2)} & \Hom(G,N_1)
\ar[d]\ar[r]^{\Hom(G,\alpha_1)} & \Hom(G,N_0) \ar[r]\ar[d] & 0 \\
0 \ar[r] & \Hom(G,N_2') \ar[r]^{\Hom(G,\alpha_2')} & \Hom(G,N_1')
\ar[r]^{\Hom(G,\alpha_1')} & \Hom(G,N_0') \ar[r] & 0 
}$$
where each vertical map is  the appropriate $\Hom(M,\beta_j)$.
\end{proof}

\begin{prop} \label{thm13}
Fix an $R$-module $M$ and a short exact sequence of $R$-modules
$0\to N_2\xra{\alpha_2}   N_1\xra{\alpha_1} N_0\to 0$ 
such that, for each $G\in\GF(R)$, the tensored sequence
\[ 0\to G\otimes_R N_2\to G\otimes_R N_1\to G\otimes_R N_0\to 0 \]
is exact.
(For instance, this holds when the sequence is split exact or when 
$N_0$ has finite injective dimension.)
There is a long exact sequence 
$$
\cdots
\Tor^{\GF}_{n}(M,N_2)
\to
\Tor^{\GF}_{n}(M,N_1)
\to
\Tor^{\GF}_{n}(M,N_0)
\xra{\varpi^{n}_{MN}}
\Tor^{\GF}_{n+1}(M,N_2)
\cdots
$$
where the unmarked maps are 
$\Tor^{\GF}_{n}(M,\alpha_2)$ and 
$\Tor^{\GF}_{n}(M,\alpha_1)$, respectively.
This sequence is natural in both arguments as in 
Proposition~\ref{thm01}.
\end{prop}

\begin{proof}
When $N_0$ has finite injective dimension, one has $\Tor^R_1(G,N_0)=0$ 
for each $G\in\GF(R)$ by~\cite[(3.13)]{H}, so 
the tensored sequence is exact in this case.  The remainder of the 
proof is similar to that of Proposition~\ref{thm01}.
\end{proof}

\begin{prop} \label{thm02}
Let $0 \to   M\xra{\mu}
M'\xra{\mu'} M''\to 0$ be a GF-exact sequence
of $R$-modules and  $N$ an
$R$-module.  There are exact sequences
\begin{align*}
\cdots                
\Ext_{\GF}^n (M'',N)   \rightarrow
\Ext_{\GF}^{n}(M',N)   \rightarrow
\Ext_{\GF}^n(M,N)      \xra{\varrho^n_{MN}}
\Ext_{\GF}^{n+1}(M'',N) 
  \cdots
\\ 
\cdots
\Tor_{n+1}^{\GF} (M'',N)
\xra{\varsigma_{n+1}^{MN}}  \Tor_{n}^{\GF}(M,N)
\rightarrow  \Tor_{n}^{\GF}(M',N)
\rightarrow  \Tor_{n}^{\GF}(M'',N)
\cdots
\end{align*}
with unmarked maps 
$\Ext_{\GF}^{n}(\mu',N)$,
$\Ext_{\GF}^{n}(\mu,N)$,
$\Tor^{\GF}_{n}(\mu,N)$, and 
$\Tor^{\GF}_{n}(\mu',N)$, respectively.
These sequences are natural in both arguments as in 
Proposition~\ref{thm01}.
\end{prop}

\begin{proof}
Fix proper Gorenstein flat resolutions $\gamma\colon G\to M$ and 
$\gamma''\colon G''\to M''$.  Lemma~\ref{lem03} yeilds a commutative 
diagram
$$\xymatrix{
0 \ar[r] & G\ar[d]_{\gamma} \ar[r]^{\widetilde{\mu}} & G' \ar[r]^{\widetilde{\mu}'} \ar[d]_{\gamma'}
& G''  \ar[r]\ar[d]_{\gamma''} & 0 \\
0 \ar[r] & M \ar[r]^{\mu} & M' \ar[r]^{\mu'} 
& M'' \ar[r] & 0}$$ 
where $\gamma'$ is a proper Gorenstein flat resolution and the top row is 
degreewise split exact.  Applying $\Hom_R(-,N)$ and $- \otimes_R N$ to 
the top row and taking (co)homology
yields the desired sequences.
\end{proof}

The next theorem is one of the main results of this section. It shows that the vanishing of $\Ext^n_{\GF}(M,-)$ measures the Gorenstein flat dimension 
of $M$.

\begin{thm} \label{thm14}
For an $R$-module $M$ and an integer
$g\geq 0$, the following 
conditions are equivalent:
\begin{enumerate}[\rm(i)]
\item \label{item01}
$\Gfd_R M\leqslant g$;
\item \label{item02}
$\Ext_{\GF}^n (M,-) = 0$ for each $n>g$;
\item \label{item20}
$\Ext_{\GF}^{g+1}(M,-) = 0$;
\item \label{item21}
For each Gorenstein flat resolution
$G\to  M$ and each integer $n\geq g$ the cokernel
$\Omega^nG$ is Gorenstein flat;
\item \label{item22}
M admits a proper Gorenstein flat resolution
$G\to  M$ such that the cokernel
$\Omega^gG$ is Gorenstein flat.
\end{enumerate}
\end{thm}

\begin{proof}
The implications 
\eqref{item01}$\implies$\eqref{item02}$\implies$\eqref{item20} 
are straightforward.  The 
implications 
\eqref{item22}$\implies$\eqref{item21}$\implies$\eqref{item01} are 
standard results, using the fact $M$ admits a proper Gorenstein flat 
resolution.  Thus, it remains to verify
\eqref{item20}$\implies$\eqref{item22}.

Assume that $\Ext_{\GF}^{g+1}(M,-) = 0$ and fix a
proper Gorenstein flat resolution
$G\to  M$.  
Consider the
follow exact sequence:
\begin{equation} \label{eq05}
0
\rightarrow \Omega^{g+1}G
\rightarrow G_{g+1}
\rightarrow \Omega^g G
\rightarrow   0
\end{equation}
Since the resolution $G$ is proper, it is routine to check 
that the sequence~\eqref{eq05} is  $\GF$-exact.
Standard dimension-shifting arguments give an isomorphism
$$\Ext_{\GF}^1 (\Omega^g G, \Omega^{g+1} G)\cong \Ext_{\GF}^{g+1}(M,
\Omega^{g+1} G) = 0$$
so Proposition~\ref{thm01} yields an exact sequence
\[ 0\to\Hom_R(\Omega^{g} G,\Omega^{g+1} G)\to 
\Hom_R(\Omega^{g} G, G_{g+1})\to
\Hom_R(\Omega^{g} G,\Omega^{g} G)\to 0 \]
which shows that sequence~\eqref{eq05} is split exact.  Since 
$G_{g+1}$ is Gorenstein flat, it follows from~\cite[(3.13)]{H}
that the same is true of
$\Omega^{g} G$, as desired.  
\end{proof}

We have a slightly weaker result for $\Tor^{\GF}$.

\begin{prop} \label{prop01}
Consider the following 
conditions on an $R$-module $M$ and an integer $g\geq 0$:
\begin{enumerate}[\rm(i)]
\item \label{item23}
$\Gfd_R M\leq g$;
\item \label{item24}
$\Tor^{\GF}_n (M,-) = 0$ for each $n>g$;
\item \label{item25}
$\Tor^{\GF}_{g+1}(M,-) = 0$.
\end{enumerate}
The implications 
\eqref{item23}$\implies$\eqref{item24}$\implies$\eqref{item25} hold.  
When $\Gfd_R M$ is finite, the conditions
\eqref{item23}--\eqref{item25} are equivalent.
\end{prop}

\begin{proof}
The only nontrivial thing to check is the implication
\eqref{item25}$\implies$\eqref{item23} under the hypothesis 
$\Gfd_R M<\infty$.  
We  show that $\Omega^g G$ is Gorenstein flat where $G\to M$ is a 
proper Gorenstein flat resolution.  Fix an injective $R$-module $I$.
Since $\Gfd_R \Omega^g G$ is finite, it 
sufffices by~\cite[(3.14)]{H}
to show that $\Tor^R_1(\Omega^g G,I)=0$.
Consider again the GF-exact sequence
\begin{equation} \label{eq08}
0
\rightarrow \Omega^{g+1}G
\rightarrow G_{g+1}
\rightarrow \Omega^g G
\rightarrow   0
\end{equation}
Since $G_{g+1}$ is Gorenstein flat, we have $\Tor^R_i(G_{g+1},I)=0$ for all $i\geqslant 1$.  
A piece of the long 
exact sequence gotten from applying $\Tor^R_i(-,I)$ to~\eqref{eq08} has the form
\[ 0\to \Tor^R_1(\Omega^g G,I)\to \Omega^{g+1}G\otimes_R I
\rightarrow G_{g+1}\otimes_R I
\rightarrow \Omega^g G\otimes_R I
\rightarrow   0. \]
Thus, it suffices to show that the map $\Omega^{g+1}G\otimes_R I
\rightarrow G_{g+1}\otimes_R I$ is injective.  
Apply Proposition~\ref{thm02} to~\eqref{eq08} in order to 
obtain the exact sequence
\[ 0 = \Tor^{\GF}_1(\Omega^g G,I)\to \Omega^{g+1}G\otimes_R I
\rightarrow G_{g+1}\otimes_R I
\rightarrow \Omega^g G\otimes_R I
\rightarrow   0 \]
which shows that the desired map is injective.
\end{proof}

\begin{obs} \label{rmk02}
Fix a projective resolution
$P\xra\pi M$ and a proper Gorenstein flat resolution
$G\xra\gamma M$.  
Lemma~\ref{lem01}\eqref{item26} yields a commutative 
diagram
$$\xymatrix{
P\ar[d]_{\vf} \ar[r]^{\pi} & M\ar[d]^= \\
G \ar[r]^{\gamma} & M.
}$$
For each $R$-module $N$, apply $\Hom_R(-, N)$ and
$-\otimes_R N$ to $\vf$ and take (co)homology to obtain 
homomorphisms
$$
\eta^n_{MN}\colon \Ext_{\GF}^n(M,N) \to \Ext_R^{n} (M,N) \qquad 
\delta_n^{MN}\colon \Tor_n^{R}(M,N)\to \Tor_n^{\GF}(M,N).
$$
The uniqueness statements in Lemma~\ref{lem01} imply that these maps are independent of 
the resolutions and lifts chosen.

Using Lemma~\ref{lem01}
it is routine to show that these maps are compatable with the 
usual maps:  Given $R$-module homomorphisms $\alpha\colon M\to M'$ and $\beta\colon N\to N'$, there 
exist commutative diagrams
$$\xymatrix{
\Ext^n_{\GF}(M,N) \ar[d]\ar[r]^{\eta^n_{MN}} & \Ext^n_{R}(M,N)\ar[d] &  
\Ext^n_{\GF}(M',N) \ar[d]\ar[r]^{\eta^n_{M'N}} & \Ext^n_{R}(M',N)\ar[d] \\
\Ext^n_{\GF}(M,N') \ar[r]^{\eta^n_{MN'}} & \Ext^n_{R}(M,N') &  
\Ext^n_{\GF}(M,N) \ar[r]^{\eta^n_{MN}} & \Ext^n_{R}(M,N) 
}$$
and similarly for $\Tor$.
Furthermore, Lemma~\ref{lem03} provides similar compatibility with the connecting 
homomorphisms in the long exact sequences constructed in 
Propositions~\ref{thm01}--\ref{thm02} and those coming from the long 
exact sequences on the usual derived functors.  The interested reader 
is encouraged to build the diagrams and verify their commutativity.
We  check one of these results explicitly, as it is used in the 
sequel.
\end{obs}

\begin{prop} \label{prop02}
Let $0 \to   M\xra{\mu}
M'\xra{\mu'} M''\to 0$ be a GF-exact sequence
of $R$-modules, and let $N$ be an
$R$-module.  The following diagram commutes
$$\xymatrix@C=6.5mm{
\cdots 
 \Tor_{n+1}^{R} (M'',N)
\ar[d]_{\delta_{n+1}^{M''N}}\ar[r]
& \Tor_{n}^{R}(M,N)
\ar[r]\ar[d]_{\delta_{n}^{MN}} &  \Tor_{n}^{R}(M',N)
\ar[r]\ar[d]_{\delta_{n}^{M'N}} &  \Tor_{n}^{R}(M'',N)\ar[d]_{\delta_{n}^{M''N}}
 \cdots \\
\cdots  \Tor_{n+1}^{\GF} (M'',N)
\ar[r]
& \Tor_{n}^{\GF}(M,N)
\ar[r] &  \Tor_{n}^{\GF}(M',N)
\ar[r] &  \Tor_{n}^{\GF}(M'',N)
 \cdots 
}$$
where the top row is the usual long exact sequence in Tor,
and the bottom row is the long exact sequence from Theorem~\ref{thm02}.
\end{prop}

\begin{proof}
Apply $-\otimes_R N$ to the diagram given in
Lemma~\ref{lem03} and take homology to obtain the desired diagram.
\end{proof}

The next theorem shows that the bijectivity of the above maps 
characterizes the modules of finite flat dimension.

\begin{thm} \label{thm03}
Let  $M$ and $N$ be $R$-modules.
\begin{enumerate}[\rm(a)]
\item \label{item04}
If $\id_R N<\infty$, then 
$\delta^{MN}_n\colon\Tor_{n}^{R} (M,N) \to \Tor_n^{\GF} (M,N)$ is an 
isomorphism for each $n\geq 0$.
\item \label{item05}
If $N$ is cotorsion and $\fd_R N<\infty$, then 
$\eta^n_{MN}\colon \Ext_{\GF}^n(M,N) \to \Ext_R^{n} (M,N)$ is an 
isomorphism for each $n\geq 0$.
\item \label{item03}
If $\Gfd_R M$ is finite, then the following conditions are equivalent:
\begin{enumerate}[\rm(i)]
\item
$\fd_R M<\infty$;
\item
$\delta^{M-}_n\colon\Tor_{n}^{R} (M,-) \to \Tor_n^{\GF} (M,-)$ is an 
isomorphism of functors for each $n\geq 0$;
\item
$\delta^{M-}_n\colon\Tor_{n}^{R} (M,-) \to \Tor_n^{\GF} (M,-)$ is an 
isomorphism of functors for each $n>\Gfd_R M$;
\item
$\delta^{M-}_n\colon\Tor_{n}^{R} (M,-) \to \Tor_n^{\GF} (M,-)$ is an 
isomorphism of functors for $n\gg 0$.
\end{enumerate}
\end{enumerate}
\end{thm}

\begin{proof}
\eqref{item04}
We prove this result 
by induction on $n$.  The case $n=0$ is covered in 
Proposition~\ref{thm12}, so assume $n>0$.
Let $G\to M$ be a proper Gorenstein flat resolution 
and consider the GF-exact sequence
\begin{equation} \label{eq09}
0\to\Omega^1 G\to G_1\to M\to 0.
\end{equation}
Since $G_1$ is Gorenstein flat, there are equalities
\[ \Tor^{\GF}_i(G_1,N)=0=\Tor^R_i(G_1,N) \]
for each $i\geq 1$, where the first is from~\cite[(3.14)]{H} and the 
second is by Propsoition~\ref{prop01}.  
Applying Proposition~\ref{prop02} to~\eqref{eq09} yields a commutative 
diagram
$$\xymatrix@C=7mm{
0=\Tor^{R}_n(G_1,N) \ar[r]\ar[d] & \Tor^{R}_n(M,N) \ar[r] \ar[d]_{\delta_n^{MN}} & 
\Tor^{R}_{n-1}(\Omega^1G,N) \ar[r]\ar[d]_{\delta_{n-1}^{\Omega^1GN}}^{\cong} & \Tor^{R}_{n-1}(G_1,N)\ar[d]_{\delta_{n-1}^{G_1N}}^{\cong} \\
0=\Tor^{\GF}_n(G_1,N) \ar[r] & \Tor^{\GF}_n(M,N) \ar[r] & 
\Tor^{\GF}_{n-1}(\Omega^1G,N) \ar[r] & \Tor^{\GF}_{n-1}(G_1,N)
}$$
where the rows are exact.
The two right-hand vertical maps are bijective by induction, and a 
routine diagram chase shows that $\delta_n^{MN}$ is as well.

\eqref{item05}
The proof is dual to that of part~\eqref{item04}, using 
Theorem~\ref{thm14} and~\cite[(3.22)]{H}.

\eqref{item03}
(i)$\implies$(ii)
When
$\fd_R M$ is finite, we claim that $M$ has a
proper $\GF$-resolution which is also a flat resolution. 
(Once this is shown, one sees that this resolution
computes both $\Tor^R(M,-)$ and $\Tor^{GF}(M,-)$, so the natural morphism between them is an isomorphism.)
Set $g=\Gfd_RM=\fd_RM<\infty$.

From~\cite[3.23]{H}, we have an exact sequence
$$
0\to   K\to   G\xra{\phi}   M\to   0
$$
such that
$\phi$ is a
Gorenstein flat precover and $K$ is cotorsion with
$\fd_R K= g-1$.
Since $K$ and $M$ have finite flat dimension, so does $G$.
Thus, the fact that $G$ is Gorenstein flat implies that $G$ is flat.
If $g=0$, then $K=0$, so the isomorphism $G\xra\phi M$ gives the desired resolution.
If $g\geq 1$, then an induction argument shows that $K$ has a
proper $\GF$-resolution which is also a flat resolution; splice this resolution with the above exact sequence to
construct the desired resolution.

The implications (ii)$\implies$(iii)$\implies$(iv) are trivial. 
For (iv)$\implies$(i), the assumption that $\Gfd_R M$ is finite yields 
$\Tor^{\GF}_i(M,-)=0$ for all $i\gg 0$ by Proposition~\ref{prop01}.  
Thus, $\Tor^{R}_i(M,-)=0$ for all $i\gg 0$, and this implies the finiteness of $\fd_R M$.
\end{proof}

The next example shows that a version of part~\eqref{item03} of above theorem does not hold for (relative) derived functors of
Hom.

\begin{ex} \label{rmk03}
Let $R$ be an integral domain  that is not a field, with fraction field $K$. Then
$\Gfd_R K =\fd_R K=0$.
Thus, by Theorem~\ref{thm03}\eqref{item03}, we have
$\Tor_n^{\GF}(K,-) =0 $ for $n>0$.
On the other hand, we have $1\leq \pd_R K$, 
so $\Ext^1_{\GF}(K,-)=0\neq \Ext^1_{R}(K,-)$.
\end{ex}

\begin{thm} \label{thm05}
Let $M$ be an $R$-module. If $\Gfd_R M <\infty$, then
\begin{align*}
\Gfd_R M 
&= \sup \{i\mid\text{$\Ext_R^i (M,K) \neq 0$ for some $K\in\operatorname{CF}(R)$} \}\\
&= \sup \{i\mid\text{$\Ext_R^i (M,F) \neq 0$ for some cotorsion
flat $R$-module  $F$} \}\\
& =\sup \{i\mid\text{$\Ext_{\GF}^i (M,N)\neq 0$ for some
$R$-module $N$} \}.
\end{align*}
\end{thm}

\begin{proof}
Let
$g=\Gfd_R M <\infty$.
The equality
$$g=\sup \{i\mid\text{$\Ext_{\GF}^i (M,N)\neq 0$ for some
$R$-module $N$} \}$$
is by Theorem~\ref{thm14}.

By~\cite[3.14]{H}, there is an injective $R$-module $I$ such that,
$\Tor_g^R(I,M)\neq 0$.
Let $E$ be a faithfully injective $R$-module $E$, and set $F = \Hom_R (I,E)$.
Note that $F$ is flat and cotorsion.
Then we have
$$
\Ext_R^g (M,F)\cong\Hom_R (\Tor_g^R (I,M),E)  \neq 0$$ 
by a version of Hom-tensor adjointness.
This explains the first inequality in the next display, while the second one follows from the fact that every cotorsion 
flat $R$-module is in $\operatorname{CF}(R)$.
\begin{align*}
g
&\leq \sup \{i\mid\text{$\Ext_R^i (M,F) \neq 0$ for some cotorsion
flat $R$-module  $F$} \}\\
&\leq \sup \{i\mid\text{$\Ext_R^i (M,K) \neq 0$ for some $K\in\operatorname{CF}(R)$} \}
\end{align*}
We prove the next inequality by induction on $g$.
$$g\geq \sup \{i\mid\text{$\Ext_R^i (M,K) \neq 0$ for some $K\in\operatorname{CF}(R)$} \}$$
If $g=0$, then~\cite[3.22]{H} implies that $\Ext_R^i (M,K) = 0$ for all $i\geq 1$ and all $K\in\operatorname{CF}(R)$;
this gives the desired inequality in this case.
The inductive step $g\geq 1$ follows readily from a dimension-shifting argument
applied to an exact sequence
$$0\to M'\to G\to M\to 0$$
such that $G$ is Gorenstein flat and $\Gfd_RM'=g-1$.
\end{proof}

\section{Base Change} \label{sec130711a}

The next result provides certain ``approximations'' of modules of finite Gorenstein flat dimension.

\begin{thm} \label{thm04}
If $M$ is an $R$-module
with $\Gfd_R M <\infty$, then $M$ has a monic CF-preenvelope   
$\varphi \colon M \to K$
such that
$\fd_R K  = \Gfd_R M$
and
$\coker \varphi$
is Gorenstein~flat.
\end{thm}

\begin{proof}
Set $g:=\Gfd_R M$.
We proceed by cases.

For the first case, assume that $g = 0$. Then $M$ is
Gorenstein  flat, so
a piece of a complete flat resolution of $M$ provides an exact sequence
$$
0\to   M\to   Q\to   C\to   0
$$
such that $Q$ is flat and $C$ is Gorenstein flat.
Since $Q$ is flat, there is a cotorsion flat module $K$ containing $Q$, such
that the quotient $H:={K}/{Q}$ is flat. (Specifically,  $K$ is the ``pure injective
envelope'' of $Q$; see the proof of~\cite[Proposition 2.1]{enochs:mpirfm}.) 
We prove that the composition of inclusions $M\xra{\vf} K$ has the desired properties.
Set $L={K}/{M}$.
By assumption, we have
$C = {Q}/{M}$, which is a submodule of ${K}/{M} = L$ such that 
${L}/{C}\cong H$. 
Since $\Gfd_R C  = 0= \fd_R H$, we have
$\Tor_1^R (J,C) =0=\Tor_1^R (J,H)$ for all injective modules $J$ by~\cite[3.14]{H}.
From the short exact sequence
$$
0\to   C\to   L\to
H\to   0 $$ 
and the associated long exact sequence in $\Tor^R_i(J,-)$
we have $\Tor_1^R (J,L) = 0$ for all injective  $J$,
so
$L$ is Gorenstein flat. Thus,  from~\cite[3.22]{H} we have $\Ext^i_R(L,Q)=0$ for all $i\geq 1$ and all $Q\in\CF(R)$.
From the short exact sequence $$0\to M\xra{\vf}K\to L\to 0$$ and the associated long exact sequence in
$\Ext^i_R(-,Q)$, we conclude that $\vf$ is a CF-preenvelope

For the second case, assume  $g>0$. 
From~\cite[3.23]{H}, we have an exact sequence
$$
0\to   X\to   G\xra\phi   M\to   0
$$
such that
$\phi$ is a
Gorenstein flat precover and $X$ is cotorsion with
$\fd_R X= g-1$.
By the previous case, there is a second exact sequence
$$0\to G\xra{\alpha} Q\to G'\to 0$$ 
such that $\alpha$ is a CF-preenvelope where $Q$ is flat  cotorsion
and $G'$ is Gorenstein flat.
Consider the
next push-out diagram.
$$\xymatrix{
&&0\ar[d]&0\ar[d] \\
0\ar[r]
&X\ar[r]\ar[d]_=
&G\ar[r]\ar[d]
&M\ar[r]\ar[d]
&0\\
0\ar[r]
&X\ar[r]
&Q\ar[r]\ar[d]
&K\ar[r]\ar[d]
&0\\
&&G'\ar[d]\ar[r]^-=&G'\ar[d] \\
&&0&0
}$$
The middle row of this diagram shows that $K$ is  cotorsion with $\fd_R K=g$.
Since $G'$ is Gorenstein flat,  we have $\Ext^1_R(L,Q)=0$ for all  $Q\in\CF(R)$ by~\cite[3.22]{H}.
Thus, the map $M\to K$ in the right-most column of this diagram has the desired properties.
\end{proof}

The next result connects ``strict'' and proper resolutions.

\begin{prop}\label{prop130708a}
Let $M$ be an $R$-module such that $g=\Gfd_RM<\infty$.
\begin{enumerate}[\rm(a)]
\item\label{prop130708a1}
Then $M$ has an augmented Gorenstein flat resolution
$$0\to F_g\to\cdots F_1\to G_0\to M\to 0$$
such that each $F_i$ is flat cotorsion.
\item\label{prop130708a2}
Every augmented Gorenstein flat resolution
\begin{equation}\label{eq130708a}
0\to F_n\xra{\partial_n}\cdots\xra{\partial_2} F_1\xra{\partial_1} G_0\xra{\tau} M\to 0
\end{equation}
such that each $F_i$ is flat cotorsion
is proper.
\end{enumerate}
\end{prop}

\begin{proof}
\eqref{prop130708a1}
From~\cite[3.23]{H}, we have an exact sequence
$$
0\to   X\to   G_0\xra\phi   M\to   0
$$
such that
$\phi$ is a
Gorenstein flat precover and $X$ is cotorsion with
$\fd_R X= g-1$. Remark~\ref{rmk01} shows that $X$ admits a  resolution
$$0\to F_g\to\cdots\to F_1\to X\to 0$$
such that each $F_i$ is flat cotorsion.
Splice the displayed sequences to find the desired resolution.

\eqref{prop130708a2}
Since $F_n$ and $F_{n-1}$ are cotorsion, it is straightforward to show that 
$\coker(\partial_n)$ is cotorsion, and similarly for each $X_i:=\coker(\partial_{i})$ with $i\geq 2$.
Furthermore,  $X_i$ has finite flat dimension for all $i\geq 2$.
Thus, we have $\Ext^j_R(H,X_i)=0$ for all $i\geq 2$, all $j\geq 1$ and for each Gorenstein flat $R$-module $H$, by~\cite[3.22]{H}.
Thus, each of the following sequences is $\Hom_R(H,-)$ exact
\begin{gather*}
0\to X_i\to F_{i-2}\to X_{i-1}\to 0 \\
0\to X_2\to G_0\to M\to 0.
\end{gather*}
It follows that the sequence~\eqref{eq130708a} is $\Hom_R(H,-)$ exact, hence proper.
\end{proof}

The remainder of this section is devoted to studying the behavior of these constructions
along surjections $R\to R/(\mathbf x)=\ol R$ where $\mathbf x$ is an $R$-regular sequence.
For convenience we set $\overline{\GF}=\GF(\ol R)$

\begin{thm} \label{thm09}
Let $\mathbf x=x_1,\ldots,x_n\in R$ be an $R$-regular sequence and set ${\ol{(-)}} =  - \otimes_R R/(\mathbf x)$. Let $M$ be an $R$-module, and let $N$ be an $\ol R$-module.
If
$\Gfd_R M <\infty $
and $\mathbf x$ is  $M$-regular, then $\Gfd_{\ol R}\ol M\leq\Gfd_R(M)$, and for all $n$ there are isomorphisms
$\Ext_{\GF}^n (M,N) \cong\Ext_{\overline{\GF}}^n (\ol{M},N)$ and 
$\Tor_n^{\GF} (M,N) \cong\Tor_n^{\overline{\GF}}(\ol{M}, N)$.
\end{thm}

\begin{proof}
Arguing by induction on $n$, it suffices to consider the case $n=1$.

Claim: 
if $F$ is a cotorsion flat
$R$-module then $\ol F$ is a cotorsion flat $\ol R$-module.
By~\cite[(2.3)]{E2} there are injective $R$-modules $E_i$ for $i
= 1,2$ such that $F$ is a direct summand of the flat $R$-module
$\Hom_R (E_1,E_2)$. It follows that $\ol F$ is a direct summand of the 
 $\ol R$-module
\begin{align*}
\ol R \otimes_R \Hom_R (E_1,E_2) 
&\cong \Hom_R (\Hom_R(\ol R,E_1),E_2)\\
&\cong \Hom_{\ol R} (\Hom_R(\ol R,E_1),\Hom_R(\ol R,E_2)).
\end{align*}
The isomorphisms here are Hom-evaluation and  adjointness.
Since each module $\Hom_R(\ol R,E_i)$ is injective over $\ol R$,
the claim follows from~\cite[(2.3)]{E2}.

Set
$g=\Gfd_R M<\infty$. Proposition~\ref{prop130708a}\eqref{prop130708a1} provides a proper Gorensetin flat resolution $G\to M$ of the form
$$
G^{+}\quad: 0\to   F_g\to  \cdots \to
F_1\to   G_0\to   M\to   0
$$
such that each $F_i$ is cotorsion flat. 
Since $x_1$ is $R$-regular, the proof of~\cite[3.11]{sahandi:dabf} shows that
$x_1$ is weakly $G_0$-regular and $\ol{G_0}$ is  Gorenstein flat 
over $\ol R$.
Since $x_1$ is $M$-regular, weakly $G_0$-regular, and weakly $F_i$-regular for each $i$, we conclude from~\cite[(1.1.15)]{BH}
that $\ol{G^+}=\ol G^+$ is exact.
In light of the above claim, Proposition~\ref{prop130708a}\eqref{prop130708a1} implies that $\ol G$ is 
a proper Gorenstein flat resolution of
$\ol M$ over $\ol R$. In particular, we have $\Gfd_{\ol R} \ol{M}\leq g$.

Using this, the isomorphism for relative Ext follows from the next sequence.
\begin{align*}
\Ext_{\overline{\GF}}^n ({\ol M} ,N) = & \HH^n (\Hom_{\ol R} (G\otimes_R \ol R, N)) \\
  = &  \HH^n (\Hom_R (G,\Hom_{\ol R} (\ol R,N))) \\
  = &  \HH^n (\Hom_R (G,N)) \\
  = &  \Ext_{\GF}^n (M,N)
\end{align*}
The isomorphism for relative Tor
is verified similarly.
\end{proof}

\begin{obs}\label{rmk130709b}
In~\cite[3.10--3.11]{sahandi:dabf}, the following is claimed:
Let $M$ be an $R$-module, and let $\mathbf x=x_1,\ldots,x_n\in R$ be a  sequence that is $R$-regular and $M$-regular.
Set ${\ol{(-)}} =  - \otimes_R R/(\mathbf x)$.
\begin{enumerate}[\rm(a)]
\item\label{item14'}
If $\Td_R M<\infty$, then 
$\Td_R M = \Td_{\ol{R}}\ol M$.
\item\label{item15'}
If
$\Gfd_R M <\infty$
then
$\Gfd_R
\ol M- n=\Gfd_R M =\Gfd_{\ol{R}}
\ol M$.
\end{enumerate}
The following example shows that the first and third equalities fail, even when $R$ is a complete regular local ring.
\newcommand{\ul}{\underline}
\newcommand{\q}{\mathfrak{q}}

Let $k$ be a field and set $R=k[\![X_1,\ldots,X_a,Y_1,\ldots,Y_n,Z_1,\ldots,Z_c]\!]=k[\![\ul X,\ul Y,\ul Z]\!]$
with $a,c\geq 0$ and $n\geq 1$.
Consider the $R$-module
$$M=(R_{\p}/(\ul X)R_{\p})\oplus (R_{\q}/(\ul Z)R_{\q})$$
where  $\p=(\ul X)$ and $\q=(\ul Y,\ul Z)$.
Then the sequence $\ul Y$ is $R$-regular and $M$-regular.

It is straightforward to show that $\fd_R(R_{\p}/(\ul X)R_{\p})=a$:
indeed, the Koszul complex $K^{R_{\p}}(\ul X)$ shows that $\fd_R(R_{\p}/(\ul X)R_{\p})\leq a$,
and the isomorphism
$$\Tor^R_a(R_{\p}/(\ul X)R_{\p},R_{\p}/(\ul X)R_{\p})\cong R_{\p}/(\ul X)R_{\p}$$ implies that
$\fd_R(R_{\p}/(\ul X)R_{\p})\geq a$.
Similarly, we have $\fd_R(R_{\q}/(\ul Z)R_{\q})=c$,
so $\fd_RM=\max\{a,c\}$.

Set ${\ol{(-)}} =  - \otimes_R R/(\ul Y)$.
Then we have
$$\ol M\cong R_{\q}/(\ul Y,\ul Z)R_{\q}.$$
The next equalities follow from~\cite[(3.16)]{H},
with computations like the ones above:
\begin{gather*}
\Td_{\ol R}\ol M=\Gfd_{\ol R}\ol M=\fd_{\ol R}\ol M=c \\
\Td_RM=\Gfd_RM=\fd_RM=\max\{a,c\} \\
\Td_R\ol M=\Gfd_R\ol M=\fd_R\ol M=n+c.
\end{gather*}

It is worth noting that~\cite[3.12]{sahandi:dabf} is correct since the only part of~\cite[3.10--3.11]{sahandi:dabf}
needed for the proof is the inequality $\Gfd_{\ol R}\ol M\leq\Gfd_R(M)$
from Theorem~\ref{thm09}. Also, the first equality in item~\eqref{item15'} above holds; see Proposition~\ref{thm10} below.

The error in the proof of item~\eqref{item14'} above is as follows. 
The authors argue by induction on $\Td_RM$.
In the induction step, the authors consider an exact sequence
$$0\to K\to F\to M\to 0$$
where $F$ is flat and $\Td_RK=\Td_RM-1$.
Since $\mathbf x$ is $R$-regular and $M$-regular, it is weakly $F$-regular and weakly $K$-regular.
However, the sequence $\mathbf x$ may not be $K$-regular, that is, the module $\ol K$ may be 0.
The induced sequence 
$$0\to \ol K\to \ol F\to \ol M\to 0$$
is exact. The authors conclude by induction that $\Td_{\ol R}\ol K=\Td_RK$.
If this were true, then their desired conclusion would follow. However, one can have $\ol K=0$, which
does not allow for the desired conclusion.

The error in the proof of item~\eqref{item15'} above is similar, coupled with an application of item~\eqref{item14'}.
\end{obs}

With Remark~\ref{rmk130709b} in mind, an analysis of the proof of~\cite[3.10--3.11]{sahandi:dabf}
provides the following. Note that 
the example in Remark~\ref{rmk130709b} shows that, when $\fd_{\ol R}\ol M<\infty$,
the quantity  $\Td_{\ol R}\ol M$ 
can take on any value up to $\Td_RM$, and similarly for $\Gfd$.

\begin{prop} \label{thm10}
Let $M$ be an $R$-module, and let $\mathbf x=x_1,\ldots,x_n\in R$ be a  sequence that is $R$-regular and $M$-regular.
Set ${\ol{(-)}} =  - \otimes_R R/(\mathbf x)$.
\begin{enumerate}[\rm(a)]
\item\label{item14}
There is an inequality $\Td_{\ol R}\ol M\leq\Td_RM$, with equality holding
provided that $\Td_RM<\infty=\fd_{\ol R}\ol M$.
\item\label{item15}
One has $\Gfd_R
\ol M- n=\Gfd_{\ol R}\ol M\leq\Gfd_RM$, with equality holding in the second step when
$\Gfd_RM<\infty=\fd_{\ol R}\ol M$.
\end{enumerate}
\end{prop}

\section{Relative Homological Algebra with Flat Cotorsion Modules} \label{sec130705a}

This section is based on the next application of Theorem~\ref{thm04}.

\begin{obs} \label{rmk04}
Let $M$ be an $R$-module with $\operatorname{Gfd}_RM<\infty$.
Theorem~\ref{thm04} shows that
we can construct a proper CF coresolution  $M\to C$. 
As with the relative derived functors from Section~\ref{sec02},
the following relative derived functors
of Hom with respect to the class $\operatorname{CF}(R)$ are well-defined:
$$\text{Ext}^n_{\CF}(N,M)= \HH^n(\text{Hom}_R(N, C)).$$
\end{obs}

\begin{thm} \label{thm06}
If 
$M$ is an $R$-module with
$\operatorname{Gfd}_RM <\infty$, then 
\begin{align*}
\operatorname{CF-id}_RM
&=\operatorname{sup}\{n
\mid\text{$\operatorname{Ext}^n_{\CF}(N,M)\neq 0$ for some $R$-module $N$}\}\\
&=\operatorname{sup}\{i\mid \text{$\operatorname{Ext}^i_{\CF}(G,M)\neq
0$ for some $G\in \operatorname{GF}(R)$}\}.
\end{align*}
\end{thm}

\begin{proof}
Set $g:=\operatorname{CF-id}_RM$ and
\begin{align*}
s
&:=\operatorname{sup}\{n
\mid\text{$\operatorname{Ext}^n_{\CF}(N,M)\neq 0$ for some $R$-module $N$}\}\\
t&:=\operatorname{sup}\{i\mid \text{$\operatorname{Ext}^i_{\CF}(G,M)\neq
0$ for some $G\in \operatorname{GF}(R)$}\}.
\end{align*}
We show that $g\leq t\leq s\leq g$. The inequality $t\leq s$ is routine.

Since $\operatorname{Gfd}_RM<\infty$,
Theorem~\ref{thm04} provides a proper CF coresolution
$$0\to M\to
Q^0\xra{\partial^0}
Q^1\xra{\partial^1}\cdots$$ such that
$\operatorname{Gfd_RM=fd}_RQ^0$ and $\operatorname{fd}_RQ^j=0=\Gfd_RG_j$ for $j\geqslant
1$ where $G_j:=\Ker \partial^j$. Hence each exact sequence 
$$0\to G_j\to Q^j\to
G_{j+1}\to 0$$ is CF-coexact. A long exact sequence argument implies that 
$\operatorname{Ext}^{i+j}_{\CF}(-,M)\cong\operatorname{Ext}^i_{\CF}(-,G_j)$ for
$i>0$ and $j\geqslant 0$. 

For the inequality $g\leq t$, assume without loss of generality that $t<\infty$.
Thus, we have $\operatorname{Ext}^n_{\CF}(G_j,M)=0$ for all
$n>t$ and all $j\geq 1$. 
In particular, this implies 
$\operatorname{Ext}^{1}_{\CF}(G_{t+1}, G_t)\cong\operatorname{
Ext}^{1+t}_{\CF}(G_{t+1},M)=0$. It follows that $Q^t=G_t\oplus G_{t+1}$ and hence
$G_t\in\operatorname{CF}(R)$.
We conclude that $g\leqslant t$, as desired.

For the  inequality $s\leq g$, assume without loss of generality that 
$g<\infty$.
It follows readily that the truncation
$$0\to M\to
Q^0\to
Q^1\to\cdots\to Q^{g-1}\to G_g\to 0$$
is a proper CF coresolution. Computing $\operatorname{Ext}^i_{\CF}(-,M)$ from this resolution,
we conclude that $\operatorname{Ext}^i_{\CF}(-,M)=0$ for all $i>s$.
The inequality $s\leq g$ now follows.
\end{proof}

\begin{thm} \label{thm07}
Let $M$ be an R-module with
$\operatorname{Gfd}_RM<\infty$, and let $p\in\mathbb N$ be given. 
\begin{enumerate}[\rm(a)]
\item \label{item06} 
The module $M$ is $p$-cotorsion with $\operatorname{fd}_RM<\infty$ if and only
if
$\operatorname{Ext}^n_{\CF}(G,M)=0=\operatorname{Ext}^n_R(G, M)$
for
$n>p$ and each $G\in \operatorname{GF}(R)$.
\item \label{item07} 
The module $M$ is strongly $p$-cotorsion with $\operatorname{fd}_RM<\infty$ if
and only if one has
$\operatorname{Ext}^n_{\CF}(N,M)=0=\operatorname{Ext}^n_R(N,M)$ for
$n>p$ and for each R-module $N$ with
$\operatorname{Gfd}_RN<\infty$.
\end{enumerate}
\end{thm}

\begin{proof}
We recycle the notation $M\to Q$ and $G_j$ from the proof of Theorem~\ref{thm06}.

\eqref{item06}
For the forward implication, assume that $M$ is $p$-cotorsion with
$\operatorname{fd}_RM<\infty$. Then, in the CF-coresolution $M\to Q$, for $j\geqslant 1$
the module $G_j$  is flat. From~\cite[(3.18)]{H}, we have
$\Ext^i_R(G,Q^j)=0$ for all $G\in\GF(R)$, all $i\geq 1$, and all $j\geq 0$.
Thus, a long exact sequence argument shows that
$\operatorname{Ext}^{i+j}_R(G,M)\cong
\operatorname{Ext}^{i}_R(G, G_j)$
for $j\geqslant 0$
and $i>0$. 
The fact that  $M$ is
$p$-cotorsion implies that $\operatorname{Ext}^1_R(G_{p+1}, G_p)\cong
\operatorname{Ext}^{p+1}_R(G_p,M)=0$. Thus, the following exact sequence  splits
$$0\to G_p \to Q^p \to
G_{p+1} \to 0.$$
We conclude that $G_p$ is cotorsion flat and
hence $g:=\operatorname{CF-id}_RM\leqslant p$. Theorem~\ref{thm06} implies that
$\operatorname{Ext}^n_{\CF}(G,M)=0$ for $n>p$. 
Dimension-shifting again, for $n>p\geq g$, we find that that
$\operatorname{Ext}^{n}_R(G,M)\cong\operatorname{
Ext}^{n-g}_R(G,G_g)=0$ for all $G\in\GF(R)$, as desired.

For the  converse, 
assume that
$\operatorname{Ext}^n_{\CF}(G,M)=0=\operatorname{Ext}^n_R(G, M)$
for
each $n>p$ and each $G\in \operatorname{GF}(R)$.
Theorem~\ref{thm06} implies that $\operatorname{CF-id}_RM\leqslant p$. In particular, we have
$\operatorname{fd}_RM<\infty$ and $\operatorname{Ext}^n_R(F,M)=0$
for  $n>p$ and each flat $R$-module $F$.

\eqref{item07}
For the forward implication, assume that $M$ is strongly $p$-cotorsion with $\operatorname{fd}_RM<\infty$,
and let $N$ be an $R$-module with $g:=\operatorname{Gfd}_RN<\infty$.
If $g=0$, then the desired conclusion follows from part~\eqref{item06}.
So assume $g>0$. By Theorem~\ref{thm04} there is an exact
sequence 
$$0\to N\to K\to
G\to 0$$ 
with $\operatorname{fd}_RK=g$, such that $K$ is cotorsion and
$G\in \GF(R)$. For $n>p$, 
the fact that $M$ is strongly $p$-cotorsion implies that $\operatorname{Ext}^n_R(K,M)=0$.
Part~\eqref{item06} implies that $\operatorname{Ext}^{n+1}_R(G,M)=0$, so the 
long exact sequence in $\Ext_R(-,M)$ implies that $\operatorname{Ext}^n_R(N,M)=0$.
The proof of part~\eqref{item06} shows that $\operatorname{CF-id}_RM\leqslant p$, thus
by Theorem~\ref{thm06} we have $\operatorname{Ext}^n_{\CF}(N,M)=0$ for $n>p$. 

The
converse is proved as in part~\eqref{item06}.
\end{proof}

\begin{cor} \label{cor02}
Let $M$ be an R-module with $\operatorname{fd}_RM<\infty $. 
\begin{enumerate}[\rm(a)]
\item \label{item08}
If $M$ is cotorsion and $G\in \GF(R)$, then
$\operatorname{Ext}^n_{\GF}(G,M)\cong\operatorname{Ext}^n_{\CF}(G,M)\cong
\operatorname{Ext}^n_R(G,M)=0$ for all $n\geq 1$. 
\item \label{item09}
If $M$ is strongly cotorsion, then one has
$\operatorname{Ext}^n_{\GF}(N,M)\cong\operatorname{Ext}^n_{\CF}(N,M)\cong
\operatorname{Ext}^n_R(N,M)=0$ for all $n\geq 1$ and all $R$-modules $N$ with $\operatorname{Gfd}_RN<\infty$.
\end{enumerate}
\end{cor}

\begin{proof}
This follows from Theorems~\ref{thm03}\eqref{item05} and~\ref{thm07}.
\end{proof}

We conclude by showing that Rid is a refinement of Gid
over almost Cohen-Macaulay local rings.

\begin{thm}\label{thm08}
Let $M$ be an $R$-module.
\begin{enumerate}[\rm(a)]
\item\label{item130711a}
One has $\operatorname{Rid}_RM\leqslant\operatorname{Gid}_RM$.
\item\label{item130711b}
If $\operatorname{Gid}_RM<\infty$, then equality holds in part~\eqref{item130711a} when
\begin{enumerate}[\rm(1)]
\item $M$ is Gorenstein injective, or
\item $R$ is local such that $\dim R-\depth R\leqslant 1$.
\end{enumerate}
\end{enumerate}
\end{thm}

\begin{proof}
Assume without loss of generality that $g:=\operatorname{Gid}_RM<\infty$.
We argue  by induction on $g$.

For the base case, assume that
$g=0$; we show that $\operatorname{Rid}_RM= 0$. In this case, $M$ is non-zero and has a complete injective
resolution 
$$E=\cdots\to E^{-2}\xra{\partial^{-2}}  E^{-1}
\xra{\partial^{-1}} 
E^0\xra{\partial^0}  E^{1}
\to\cdots.$$ 
Set $N_i=\Ker\partial^{-i}$ for
each $i$. For each $R$-module $X$ and each $n,d\geq 1$, dimension-shifting implies that
$\operatorname{Ext}^n_R(X,M)\cong
\operatorname{Ext}^{n+d}_R(X,N_d)$. 
Assume
that $\operatorname{pd}_RX<\infty$. 
Then, using $d\geq\pd_RX$, we find that $\operatorname{Ext}^n_R(X,M)=0$
for all $n\geq 1$, so $\operatorname{Rid}_RM\leq 0$. The inequality $\operatorname{Rid}_RM\geq 0$
is from~\cite[(5.11)]{Chr}.

For the induction step, assume
$g\geqslant 1$. By~\cite[2.15]{H} there is an exact sequence 
$$0\to M\to G\to K\to 0$$
such that $G\in
\GI(R)$  and $\operatorname{id}_RK=g-1$.
Consider an $R$-module $X$ with $\operatorname{pd}_RX<\infty$.
The base case, coupled with the long exact sequence in $\Ext_R^i(X,-)$ implies that
$\operatorname{Ext}^i_R(X,M)=0$ for 
$i>g$, so $\operatorname{Rid}_RM\leqslant g$. 

To complete the proof, assume that $R$ is local such that $\dim R-\depth R\leqslant 1$.
Part of a complete injective resolution of $G$ 
yields an exact sequence 
$$0\to H\to E\to G\to 0$$
such
that $E$ is injective and $H$ is Gorenstein injective.
Identify $G$ with $E/H$ so that the submodule $M\subseteq G$ is of the form 
$M= N/H$ for some submodule $N\subseteq E$.
It follows that we have $E/N\cong
G/M\cong K$, hence the next exact sequence.
$$0\to
N\to E\to K\to 0.$$ 
Since
$\operatorname{id}_RK=g-1$, we have $\operatorname{id}_RN= g$. 
Since $\dim R -\depth R\leq 1$, we conclude from~\cite[(5.13)]{Chr} that
$\operatorname{Rid}_RN=g$, so there is an $R$-module $X$ with
$\operatorname{pd}_RX<\infty$ such that $\operatorname{Ext}^g_R(X,N)\neq
0$. 
Apply the base case to $H$, and use the long exact sequence in $\Ext^i_R(X,-)$
associated to the sequence
$$0\to H\to N\to M\to 0$$
to conclude that $\operatorname{Ext}^g_R(X,M)\neq
0$. It follows that $\operatorname{Rid}_RM\geq g$, as desired.
\end{proof}

\begin{cor} \label{cr03}
Let $M$ be an $R$-module such that $\operatorname{Gid}_RM<\infty$.
\begin{enumerate}[\rm(a)]
\item \label{item10}
If $\dim R<\infty$ and $\operatorname{Gid}_RM\leqslant p$, then $M$ is
strongly $p$-cotorsion.
\item \label{item11}
If $R$ is local  with $\dim R-\depth R\leqslant 1$,
then M is strongly p-cotorsion if and only if
$\operatorname{Gid}_RM\leqslant p$.
\end{enumerate}
\end{cor}

\begin{proof}
When $\dim R<\infty$, we know from~\cite{J,R}
that an $R$-module has finite flat dimension if and only if it has finite projective dimension.
Thus, the desired conclusions follow from Theorem~\ref{thm08}.
\end{proof}


\begin{thebibliography}{10}

\bibitem{aldrich:dfhrfc}
Stephen~T. Aldrich, Edgar~E. Enochs, and Juan~A. Lopez~Ramos, \emph{Derived
  functors of {H}om relative to flat covers}, Math. Nachr. \textbf{242} (2002),
  17--26. \MR{1916846 (2003e:16004)}

\bibitem{auslander:htmcma}
M.~Auslander and R.-O.\ Buchweitz, \emph{The homological theory of maximal
  {C}ohen-{M}acaulay approximations}, M\'em. Soc. Math. France (N.S.) (1989),
  no.~38, 5--37, Colloque en l'honneur de Pierre Samuel (Orsay, 1987).
  \MR{1044344 (91h:13010)}

\bibitem{AvM}
L.~L. Avramov and A.\ Martsinkovsky, \emph{Absolute, relative, and {T}ate
  cohomology of modules of finite {G}orenstein dimension}, Proc. London Math.
  Soc. (3) \textbf{85} (2002), 393--440. \MR{2003g:16009}

\bibitem{bican:amhfc}
L.~Bican, R.~El~Bashir, and E.~Enochs, \emph{All modules have flat covers},
  Bull. London Math. Soc. \textbf{33} (2001), no.~4, 385--390. \MR{1832549
  (2002e:16002)}

\bibitem{BH}
W.\ Bruns and J.\ Herzog, \emph{Cohen-{M}acaulay rings}, revised ed., Studies
  in Advanced Mathematics, vol.~39, University Press, Cambridge, 1998.
  \MR{1251956 (95h:13020)}

\bibitem{Chr}
L.~W. Christensen, H.-B.\ Foxby, and A.\ Frankild, \emph{Restricted homological
  dimensions and {C}ohen-{M}acaulayness}, J. Algebra \textbf{251} (2002),
  no.~1, 479--502. \MR{1900297 (2003e:13022)}

\bibitem{EM}
S.\ Eilenberg and J.~C. Moore, \emph{Foundations of relative homological
  algebra}, Mem. Amer. Math. Soc. No. \textbf{55} (1965), 39. \MR{0178036 (31
  \#2294)}

\bibitem{E1}
E.~E. Enochs, \emph{Injective and flat covers, envelopes and resolvents},
  Israel J. Math. \textbf{39} (1981), no.~3, 189--209. \MR{636889 (83a:16031)}

\bibitem{E2}
\bysame, \emph{Flat covers and flat cotorsion modules}, Proc. Amer. Math. Soc.
  \textbf{92} (1984), no.~2, 179--184. \MR{754698 (85j:13016)}

\bibitem{enochs:mpirfm}
\bysame, \emph{Minimal pure injective resolutions of flat modules}, J. Algebra
  \textbf{105} (1987), no.~2, 351--364. \MR{873670 (88f:13001)}

\bibitem{EJ1}
E.~E. Enochs and O.~M.~G. Jenda, \emph{Gorenstein injective and projective
  modules}, Math. Z. \textbf{220} (1995), no.~4, 611--633. \MR{1363858
  (97c:16011)}

\bibitem{EJ3}
\bysame, \emph{Relative homological algebra}, de Gruyter Expositions in
  Mathematics, vol.~30, Walter de Gruyter \& Co., Berlin, 2000. \MR{1753146
  (2001h:16013)}

\bibitem{enochs:egfc}
E.~E. Enochs, O.~M.~G. Jenda, and J.~A. Lopez-Ramos, \emph{The existence of
  {G}orenstein flat covers}, Math. Scand. \textbf{94} (2004), no.~1, 46--62.
  \MR{2032335 (2005b:16005)}

\bibitem{enochs:gf}
E.~E. Enochs, O.~M.~G. Jenda, and B.~Torrecillas, \emph{Gorenstein flat
  modules}, Nanjing Daxue Xuebao Shuxue Bannian Kan \textbf{10} (1993), no.~1,
  1--9. \MR{95a:16004}

\bibitem{F3}
H.-B.\ Foxby, \emph{Auslander-{B}uchsbaum equalities}, Proceedings of the Amer.
  Math. Soc. 904th Meeting, Kent State Univ., Amer. Math. Soc., Providence,
  1995, pp.~759--760.

\bibitem{H2}
H.\ Holm, \emph{Gorenstein derived functors}, Proc. Amer. Math. Soc.
  \textbf{132} (2004), no.~7, 1913--1923. \MR{2053961 (2004m:16009)}

\bibitem{H}
\bysame, \emph{Gorenstein homological dimensions}, J. Pure Appl. Algebra
  \textbf{189} (2004), no.~1, 167--193. \MR{2038564 (2004k:16013)}

\bibitem{J}
C.~U. Jensen, \emph{On the vanishing of
  {$\underset{\longleftarrow}{\lim}^{(i)}$}}, J. Algebra \textbf{15} (1970),
  151--166. \MR{0260839 (41 \#5460)}

\bibitem{R}
D.~Rees, \emph{The grade of an ideal or module}, Proc. Cambridge Philos. Soc.
  \textbf{53} (1957), 28--42. \MR{0082967 (18,637c)}

\bibitem{sahandi:dabf}
P.~Sahandi and T.~Sharif, \emph{Dual of the {A}uslander-{B}ridger formula and
  {GF}-perfectness}, Math. Scand. \textbf{101} (2007), no.~1, 5--18.
  \MR{2353238 (2008j:13029)}

\bibitem{sather:gcac}
S.\ Sather-Wagstaff, T.\ Sharif, and D.\ White, \emph{Gorenstein cohomology in
  abelian categories}, J.\ Math.\ Kyoto Univ. \textbf{48} (2008), no.~3,
  571--596. \MR{2511051}

\bibitem{sather:crct}
\bysame, \emph{Comparison of relative cohomology theories with respect to
  semidualizing modules}, Math. Z. \textbf{264} (2010), no.~3, 571--600.
  \MR{2591820}

\bibitem{sather:abc}
\bysame, \emph{A{B}-contexts and stability for {G}orenstein flat modules with
  respect to semidualizing modules}, Algebr. Represent. Theory \textbf{14}
  (2011), no.~3, 403--428. \MR{2785915}

\bibitem{Xu}
J.\ Xu, \emph{Flat covers of modules}, Lecture Notes in Mathematics, vol. 1634,
  Springer-Verlag, Berlin, 1996. \MR{1438789 (98b:16003)}

\end{thebibliography}
\providecommand{\bysame}{\leavevmode\hbox to3em{\hrulefill}\thinspace}
\providecommand{\MR}{\relax\ifhmode\unskip\space\fi MR }
\providecommand{\MRhref}[2]{%
  \href{http://www.ams.org/mathscinet-getitem?mr=#1}{#2}
}
\providecommand{\href}[2]{#2}

\end{document}